\documentclass[reqno,12pt]{amsart}
\usepackage[all,line,poly,rotate,ps,dvips]{xy}
\SelectTips{cm}{}

\usepackage{amssymb,eucal,graphicx}
\usepackage{enumerate}

\numberwithin{equation}{section}
 
\newtheorem{theorem}[equation]{Theorem}

\newtheorem{lemma}[equation]{Lemma}
\newtheorem{proposition}[equation]{Proposition}

\theoremstyle{definition}

\newtheorem{definition}[equation]{Definition}
\newtheorem{remark}[equation]{Remark}

\theoremstyle{remark}

\newcommand{\C}{\mathbb{C}}

\newcommand{\Pp}{\mathbb{P}}

\newcommand{\Ff}{\mathcal{F}}

\newcommand{\CC}{\mathcal{C}}

\newcommand{\N}{\mathcal{N}}

\newcommand{\T}{\mathcal{T}}
\newcommand{\Oc}{\mathcal{O}}

\newcommand{\X}{\mathcal{X}}

\begin{document}
\title[On the branch curve of a general projection of a surface to a plane]%
{On the branch curve of a general projection of a surface to a plane}%
\author{C. Ciliberto, F. Flamini}

\email{cilibert@mat.uniroma2.it} \curraddr{Dipartimento di
Matematica, Universit\`a degli Studi di Roma Tor Vergata\\ Via
della Ricerca Scientifica - 00133 Roma \\Italy}

\email{flamini@mat.uniroma2.it} \curraddr{Dipartimento di
Matematica, Universit\`a degli Studi di Roma Tor Vergata\\ Via
della Ricerca Scientifica - 00133 Roma \\Italy}

\thanks{  
\\
The  authors are members of G.N.S.A.G.A.\ at
I.N.d.A.M.\ ``Francesco Severi''.}

\begin{abstract}  In this paper we prove that the branch curve of a general projection
of a surface to the plane is irreducible, with only nodes and cusps. 
\end{abstract}

\maketitle


\noindent
\centerline{{\em }}

\section{Introduction}\label{S:Intro}

A few years ago the first author wrote, in collaboration with R. Miranda and M. Teicher, the paper \cite {CMT}, in which the following theorem, assumed to be well know, was stated and used:

\begin{theorem}\label {thm:main} Let $S\subset \Pp^ r$ be a smooth, irreducible, projective surface (or a \emph{general surface} in $\Pp^ 4$).
Then the ramification curve on $S$ of a general projection of $S$ to $\Pp^ 2$ is smooth and irreducible and the branch  curve in the plane  is also irreducible and has only nodes and cusps, respectively corresponding to two simple ramification points and one double ramification point.
\end{theorem}

This means that the ramification of the general projection morphism $S\to \Pp^ 2$ is \emph{as simple as possible}. 
This result has been stated as a fact by various classical authors (see \cite {enr}), and in fact it is extremely 
useful in various aspects of the theory of surfaces, like Hilbert scheme and moduli space computations \`a la 
Enriques (see \cite {enr}, Chapter V, \S 11), braid monodromy computations 
(see, e.g., \cite {teic}, \cite {kul1}, \cite {kul2}) and in number theoretical problems concerning 
algebraic varieties (see, e.g. \cite{Fal}. 

An anonymous referee  kindly remarked that  there was  
no proof of Theorem \ref {thm:main} in the current literature,  despite good evidences 
for its thruth, given the fact that the result holds under the hypothesis that $S$ is the 
$k$--tuple Veronese embedding of a smooth surface, with $k\geq 2$ (see \cite {kulkul}; in \cite{Fal} the same result is 
proved under even stronger hypotheses, verified when $k \geq 5$). 

The purpose of the present note is to fill up this annoying gap in the literature, by giving a proof 
of Theorem \ref {thm:main}. Actually  we will prove the following more general result:

\begin{theorem}\label{thm:main2} Let $\Sigma \subset \Pp^3$ be a non--degenerate, projective surface, 
with ordinary singularities. Then the ramification curve of a general projection of $\Sigma$ to a plane from a point $p$ is irreducible, the branch curve is also irreducible and 
 has only nodes and cusps, respectively corresponding to bitangent,  and not tritangent, lines and simple 
asymptotic tangent lines to $\Sigma$ passing through $p$. \end{theorem}

As indicated in \cite[Lemma 1.4]{kulkul}, Theorem \ref{thm:main} (respectively, Theorem \ref{thm:main2})  
can be easily extended to the case in which  $S$ (respectively, $\Sigma$) has mild singularities (respectively, 
away from the ordinary singularities): see Remark \ref{rem:kulkul}. 

The proof of Theorem \ref {thm:main2} follows a pattern which is in principle not so different from the one 
in \cite {kulkul}. It consists in the study of families of multi-tangent lines to the surface $\Sigma$. The 
new tool with respect to \cite{kulkul} is the use of some basic techniques in 
projective--differential geometry, i.e. the classical theory of {\em foci}. The use of the focal toolkit 
simplifies the approach and enables us to prove the result without any additional assumption.

In \S \ref {S:NotPre1} we recall, for the reader's convenience, some basic notions and results about projections
of surfaces and the singularities that they produce. In \S \ref {S:NotPre} we recall the relations between the 
singularities of the branch curve of a general projection and the tangent lines to the surface passing through 
the centre of projection (see also \cite {kulkul}).
The upshot of the first two sections is to reduce the proof of Theorem  \ref {thm:main} to the one of 
Theorem \ref {thm:main2} and to reduce this, in turn, to the proof that the branch curve has only double points. 
In \S \ref {S:fuochi} we recall some generalities about focal schemes, a classical projective--differential 
subject revived in recent times in \cite{CS}, and successfully used in several contexts (see \cite {CC}). 
In \S \ref {S:fillingfam} the focal machinery is applied to prove Proposition \ref  {prop:foctang}, a result certainly known to
the  classics, for which however  we do not have
a suitable reference. In the short \S  \ref {S:Main} we collect all the information and give the
proof of Theorem \ref {thm:main2}  and Theorem \ref {thm:main}. 

We finish with \S \ref{S:considerations}, in which we outline the aforementioned application  \`a la 
Enriques of Theorem \ref {thm:main} to the study of the dimension of Hilbert schemes and moduli spaces of surfaces.

\section{Ordinary singularities and projections}\label{S:NotPre1}  

This section  is devoted to recall, for the reader's convenience, some basic facts
about general projections to $\Pp^ 3$ of smooth surfaces sitting in higher dimensional
projective spaces.  

\subsection{Ordinary singularities} \label {SS:ordsing} We start with a classical definition.

\begin{definition}\label{def:gensurf3} An irreducible, projective surface $\Sigma \subset \Pp^3$ is 
said to have {\em ordinary singularities} if its singular locus is either empty or it is  a 
curve $\Gamma$, called the {\em double curve} of $\Sigma$,  with the following properties:

\begin{enumerate}
\item $\Gamma$ has at most finitely many \emph{ordinary triple points}, such that
 the germ of $\Sigma$ there
is analytically equivalent to the one of the affine surface in $\C^3$ with equation 
$xyz=0$ at the origin;

\item every non--singular point of $\Gamma$ is either a {\em nodal point}, i.e., 
 the germ of $\Sigma$ there
is analytically equivalent to the one of the surface with equation 
$x^ 2-y^ 2=0$ at the origin, or a {\em pinch point}, i.e. 
 the germ of $\Sigma$ there
is analytically equivalent to the one of the surface with equation 
$x^ 2-zy^ 2=0$ at the origin;

\item for every irreducible component $\Gamma'$ of $\Gamma$, the general point of $\Gamma'$ is a 
nodal point of $\Sigma$, in particular, there are only finitely many pinch points for $\Sigma$.

\end{enumerate}
\end{definition}

\begin{remark}\label{rem:aiuto} Assume $\Sigma \subset \Pp^3$ has ordinary singularities and let  $\nu: X\to \Sigma$ be its normalization. It is immediate to see that $X$ is smooth and the line bundle 
$\nu^ *(\mathcal O_{\Pp^ 3}(1))$ is ample.

Let $\Delta=\nu^ {-1}(\Gamma)$. Then:
\begin{itemize}
\item[(i)] $\nu{\vert _{\Delta}}$ is a generically $2:1$ covering;

\item[(ii)] if $p \in \Gamma$ is a triple point, then 
$\nu^{-1}(p) := \{ p_1, p_2,p_3\}$,
where each $p_i$ is a node of $\Delta$, $1 \leq i \leq 3$. 
We will denote by $T$ the set of these nodes,
which are the only singular points of $\Delta$;

\item[(iii)] If $q \in \Gamma$ is a pinch point, then $q$ is a branch point of 
$\nu{\vert _{\Delta}}$, over which $\Delta$ is smooth.
We will denote by $\Omega$ the set of corresponding ramification points on $\Delta $.   
\end{itemize}
\end{remark}

\subsection {Projections} \label {SS:proj} It is a classical fact that surfaces with ordinary singularities occur as general projections in $\Pp^ 3$ of smooth surfaces in higher dimensional projective spaces. 
It is useful to recall the basic results on this subject. 

Let $S \subset \Pp^r$, $r\geq 3$, be a smooth, irreducible, non--degenerate
projective surface.   
For any $k <r$, we denote by 
$\varphi_k : S \to \Pp^k$ the projection of $S$ from a general linear 
subspace  of dimension $r-k-1$ of  $\Pp^r$. If $h<k$, we may
assume that $\varphi_h$ factors through $\varphi_k$. 

If $r \geq 6$, then
$\varphi_5 : S \to \Pp^5$ maps $S$ isomorphically to a smooth surface in  $\Pp^5$, 
since the secant variety ${\rm Sec}(S)$ does not fill up $\Pp^ r$. Projections
to $\Pp^ 4$ and $\Pp^ 3$ no longer preserve smoothness. 

\begin{definition}\label{def:gensurf4} (cf. \cite[Def.1]{MP}) A non-degenerate, irreducible, projective surface $\Sigma \subset \Pp^4$ is called a {\em general surface} of  $\Pp^4$ if either $\Sigma$ is smooth or the singularities of 
$\Sigma$ are at most a finite number of {\em improper double points}, i.e.  the origin of two smooth branches of $\Sigma$ with independent tangent planes. 
\end{definition}

The normalization of a general surface in $\Pp^ 4$ is smooth.

A general projection of a smooth, irreducible, non--degenerate surface $S\subset \Pp^ r$, $r\geq 5$, to 
$\Pp^ 4$ is a general surface of $\Pp^ 4$: its improper double points correspond to the
intersections of the centre of projection with ${\rm Sec}(S)$ and their number is given by the \emph{double point formula} (see \cite {Sev}, \cite {Ful}).  More precisely, one has (cf. \cite{Sev} and \cite[Thm.3]{MP}):

\begin{theorem}\label{thm:Sev} The only 
smooth surface $S\subset \Pp^ 5$ whose general projection to $\Pp^ 4$ is smooth
is  the Veronese surface of conics.  
\end{theorem}

This is the same as saying that the Veronese surface of conics  is the only smooth \emph{defective} surface $S$ in $\Pp^ 5$, i.e. such that $\dim({\rm Sec}(S))<5$. 

As for projections to $\Pp^ 3$, it was classically stated by various authors, like M. Noether, F. Enriques etc.,  that the general projection of a smooth surface to $\Pp^ 3$ has only ordinary singularities  (see, e.g., \cite {enr}, and \cite {GH2} for a modern reference). A  more precise result is the following (see  \cite[Thm.8]{MP}):

\begin{theorem}\label{thm:genpro} (The General Projection Theorem)  Let $\Sigma$ be a 
general surface of $\Pp^4$. Then a general projection of $\Sigma$ to $\Pp^3$ has ordinary singularities. 
\end{theorem}

Though not essential for us, it is worth recalling that, taking the General Projection Theorem for granted, Franchetta proved in \cite{Fra} the following result (see also \cite[Thm.5]{MP}):
 
\begin{theorem}\label{thm:Fra} Let $\Sigma$ 
be a general surface of $\Pp^4$. Then the double curve of its general projection to $\Pp^ 3$  is irreducible, unless $\Sigma$ is the projection of the Veronese surface of conics
to $\Pp^4$. 
\end{theorem}

The proof is rather simple if $\Sigma$ is the projection of a smooth surface
in $\Pp^ 5$ (see  \cite{Mum}), otherwise the argument is quite delicate.

\section{Branch curves of projections}\label{S:NotPre}

In this section we recall the relations between the singularities of 
the branch curve of a general projection and the
tangent lines to the surface meeting the centre of projection.
Most of this is essentially 
contained in  \cite {kulkul}. We dwell on this here 
in order to make this note as self--contained as possible. 

\subsection{Branch curves}\label{SS:branch}
Let $\Sigma\subset \Pp^ 3$ be an irreducible
surface with ordinary singularities and let $p\in \Pp^ 3$ be a general point. Consider the projection
$\varphi: \Sigma\to \Pp^ 2$ from $p$ to a general plane in $\Pp^ 3$.  If, as above, 
$\nu : X \to \Sigma$ is the the normalization, 
one has the commutative diagram

\begin{equation}\label{eq:diagfico}
\begin{array}{cccl}
X & \stackrel{\nu}{\longrightarrow} & \Sigma & \subset\Pp^3 \\
  & \searrow^{\psi} & \downarrow^{\varphi} & \\
  &    & \Pp^2 & 
\end{array}
\end{equation}
Let  $B \subset \Pp^2$ be the branch curve of $\psi$. We will call it the \emph{branch curve} of 
$\varphi$ as well. We denote by  $R \subset X$ the 
\emph{ramification curve} of $\psi$ and by $Z$ its image on $\Sigma$. Note that $Z$ is the
residual intersection of $\Sigma$ with its polar with respect to $p$, off the double curve $\Gamma$ (see Remark \ref {rem:pinch} below).

Let $d$ be the degree of $\Sigma$ and $g$ be the geometric genus of its general plane section.
Then, by the Riemann--Hurwitz formula, one has $\deg(B)=2(d+g-1)$.

\begin{remark}\label {rem:branchproj} {\rm If $S\subset \Pp^ r$, $r\geq 5$, is a smooth, irreducible surface, and 
$\varphi_2: S\to \Pp^ 2$ is a general projection, we can consider branch and ramification curves of $\varphi_2$. 
In view of the results in \S \ref {S:NotPre1}, this is a particular case of the previous situation. Similarly for a general projection to $\Pp^ 2$ of a general surface in $\Pp^ 4$. 
}\end{remark}

\begin{remark}\label {rem:pinch} {\rm  The singular locus scheme of a surface $\Sigma\subset \Pp^ 3$ with equation
$$f(x_0,x_1,x_2,x_3)=0$$
 is the base locus scheme of the linear system of first polars of $\Sigma$. Recall that the polar of 
 $p=[p_0,p_1,p_2,p_3]$ with respect to $\Sigma$ has equation
 $$\sum_{i=0}^ 3 p_i\frac {\partial f}{\partial x_i}=0.$$
A local computation shows that, if $\Sigma$ has ordinary singularities,  the singular locus scheme consists of the double curve $\Gamma$ with
an embedded point of lenght two at each pinch point $z$. This translates the fact that,  for any pinch point $z$, all polars are tangent to the plane which is the support of the tangent cone to $\Sigma$ at $z$.
By applying Bertini's Theorem we see that $Z$ does not contain any triple point and is smooth at each pinch point. Accordingly, with notation as in Remark \ref{rem:aiuto}, $R$ does not contain any point of $T$ and is smooth at any point of $\Omega$.
}\end{remark}

\subsection {Tangent lines}\label {SS:tangent}
Let $T_\Sigma$ be the 3--dimensional, irreducible subvariety of the Grassmannian $\mathbb G(1,3)$ of lines in $\Pp^ 3$ which is the Zariski closure of the set of all lines tangent to $\Sigma$ at a smooth point. The lines in $T_\Sigma$ are called the \emph{tangent lines} to $\Sigma$.

Given a line $\ell$ in $\Pp^ 3$, the corresponding point $[\ell]\in \mathbb G(1,3)$ sits in $T_\Sigma$ if and only if there is a point $z\in \Sigma$ such that $z\in \ell$ and $\ell$ sits in the tangent cone to $\Sigma$ at $z$. In this case we say that $\ell$ is \emph{tangent} to $\Sigma$ at $z$ and that $z$ is a \emph{contact point} of $\ell$ with $\Sigma$.

If $\ell\not\subset \Sigma$, we  denote by $\ell_\Sigma$ the  $0$--dimensional scheme cut out by $\ell$ on $\Sigma$ and by $\ell_X$ its pull--back to $X$ via $\nu$. Then $[\ell]$ sits in $T_\Sigma$ if and only if  either  $\ell\subset \Sigma$ or $\ell_X$ is not reduced.  

Let $\ell$ be  a line in $\Pp^ 3$ through the centre of projection $p$. 
Write $\ell_X=n_1x_1+\ldots+n_hx_h$,
with $n_1+\ldots+n_h=d$ and $n_1\geq \ldots \geq n_h$. The line $\ell$ is tangent to $\Sigma$ if and only if there is an $i=1,\ldots,h$ such that $n_i\geq 2$, in which case $\ell$ is tangent to $\Sigma$ at $z_i=\nu(x_i)$, and $n_i-1$ is called the \emph{contact order} of $\ell$ with the branch of $\Sigma$ corresponding to $x_i$ at $z_i$.

By the genericity
assumption, $p$ does not sit on the developable tangent surface to the double curve $\Gamma$ of $\Sigma$. This implies that there are no pairs of indices $i,j$, with $1\leq i<j\leq h$ such that $n_i\geq 2, n_j\geq 2$ and $\nu(x_i)=\nu(x_j)$. 

We will say that
$\ell$ is \emph {$k$--tangent}  to $\Sigma$ if there are distinct indices $i_1,\ldots, i_k$ such that $n_{i_j}\geq 2$, for all $j=1,\ldots, k$. One  says that $\ell$ is an \emph {asymptotic tangent line}  to $\Sigma$ if
there is an index $i=1,\ldots,h$ such that $n_i\geq 3$, in which case $\ell$ is an asymtpotic tangent line to $\Sigma$ at $z=\nu(x_i)$.  The integer $a(\ell)= \sum_{i=1}^ h(n_i-2)$ is called the \emph{asymptotic weight} of $\ell$. If  $a(\ell)=1$ the line $\ell$  is called a \emph{simple} asymtpotic tangent. 

It is immediate that $y\in \Pp^ 2$ belongs to $B$ if and only if the line
$\ell=\ell_y=\langle p,y\rangle$ is tangent to $\Sigma$. If $\ell_{X}=n_1x_1+\ldots+n_hx_h$, the integer 
$b(y)=\sum_{i=1}^ h(n_i-1)$ is called the \emph{branching weight} of $y$ and $n_i-1$ is the 
\emph{ramification weight} $r(x_i)$ of $x_i$, $i=1,\ldots, h$. If these weight are 1 or 2, we will talk about 
\emph{simple} or \emph {double} branch (respectively, ramification) points.

Accordingly $x\in X$ sits on $R$,  and then  $z=\nu(x)$ sits
on $Z$, if and only if the scheme $\ell_{z,X}$ is non--reduced at $x$. 

\subsection{More on asymptotic tangent lines}\label {SS:asympt}

In the above setting, let $x\in X$ and set $z=\nu(x)$. Let   $\mathcal L$
 be the pull back on $X$ of the linear system of planes in $\Pp^ 3$. Note that there is a unique 
 curve $C_x$ in $\mathcal L$ singular at $x$, namely the pull--back to $X$ of the plane section of $\Sigma$ with the tangent plane to the branch through $z$ corresponding to $x$. 
We will say that $x$ is a \emph{planar point} of $X$
if $C_x$ has a point of multiplicity at least 3 at $x$.  If $x$ is a planar point, then all lines through 
$z$ in the tangent plane to the branch corresponding to $x$ are asymtpotic tangent lines.

\begin{lemma}\label{lem:isolated} There are only finitely many planes in $\Pp^ 3$ cutting  $\Sigma$ in a curve with a point of multiplicity $m\geq 3$.
\end{lemma} 

\begin{proof} Suppose the assertion is not true. Then there is a 1--dimensional curve $\{C_t\}_{t\in D}$ in $\mathcal L$, 
parametrized by a disc, whose
general member has a point $\{x_t\}_{t\in D}$ of multiplicity $m\geq 3$. The tangent space to this curve  
at $t=0$ is contained in the set of curves in $\mathcal L$ having multiplicity 
at least $m-1\geq 2$ at $x_0$ (see e.g. \cite{CC2} or \cite{Ser}). 
Since there is only one such curve, namely $C_{x_0}$, we find a contradiction.
\end{proof}

By genericity, we may and will assume that the centre of projection stays off the finitely many  planes 
cutting  $\Sigma$ in a curve with a point of multiplicity $m\geq 3$.

Suppose $x$ is neither a planar point nor a point in $\Omega$. 
Then, there are only one or two asymptotic tangent lines 
through $z$ in the tangent plane to the smooth branch of $\Sigma$ corresponding to $x$:
their directions
are the zero locus of the \emph{second fundamental form} of $\Sigma$ at $x$ (see \cite {GH}). 
In geometric terms, consider the curve $C_x$, which has a double point at $x$. 
The directions of the asymptotic tangent lines at $z=\nu(x)$ are the images
via the differential $d\nu_x$
of the directions of the lines in the tangent cone to $C_x$ at $x$. 
One says that $x$, or  $z=\nu(x)$,
is a \emph{parabolic point} if the tangent cone to $C_x$ at $x$ is non--reduced.
Then, the corresponding asymptotic tangent line will be called \emph{parabolic}.

Let $P(S)$ be the Zariski closure of the set of parabolic points of $S$. 
One has  $\dim(P(S))\leq 1$. Otherwise $\Sigma$ would be a \emph{developable surfaces},
i.e. either a 
cone or the locus of tangent lines to a curve (see \cite {GH}). This is not possible, since 
$\Sigma$ has ordinary singularities. 

By the genericity of the centre of projection $p$, we conclude that:

\begin{lemma}\label{lem:one} In the above setting, no
parabolic asymptotic tangent line contains the centre of projection. 
\end{lemma}

\subsection{Multiple points of the branch curve}\label {SS:multiple}

The following proposition tells us what are the multiple points of the branch curve.

\begin{proposition}\label {prop:multiplebranch} In the above setting, the multiplicity of $B$ at $y$ is the branching weight $b(y)$.
\end{proposition}

\begin{proof} Let $\ell=\ell_y$ and $\ell_X=n_1x_1+\ldots+n_hx_h$.
Let $\pi$ be a general plane through $\ell$ and let $C$ be the pull back on $X$ of the curve
section of $\Sigma$ with $\pi$. Then $C$ is smooth and irreducible of genus $g$. The projection from $p$ induces a morphism $C\to \Pp^ 1$ of degree $d$ which has a branch point $y$, corresponding to the line $\ell$ with ramification index $\sum_{i=1}^ h(n_i-1)$. Bertini's theorem guarantees that the remaining branch points, corresponding to the intersections of $\pi$ with $B$ off $y$, are all simple. Hence their number is $2(d+g-1)-\sum_{i=1}^ h(n_i-1)=\deg(B)-b(y)$. This proves the assertion.
\end{proof}

More specifically:

\begin{proposition}\label {prop:doublebranch} In the above setting, if the branch curve has only double points, 
then $R$ is smooth and irreducible, $B$ is also irreducible and has only nodes and cusps, respectively corresponding 
to bitangent, not tritangent, and simple asymptotic tangent lines containing the centre of projection $p$.
\end{proposition}

\begin{proof} Let $y$ be a singular point of $B$, which by assumption has multiplicity 2. Suppose it corresponds to the line $\ell$ through $p$, with $\ell_X=n_1x_1+\ldots+n_hx_h$. By Proposition  \ref 
{prop:multiplebranch} we have only two possibilities:
\begin{itemize}
\item [(a)] $n_1=n_2=2$, $n_3=\ldots=n_{d-4}=1$, i.e. $\ell$ is a bitangent, not tritangent, line;
\item [(b)] $n_1=3$, $n_2=\ldots=n_{d-3}=1$, i.e. $\ell$ is a simple asymptotic tangent line.
\end{itemize}

In case (a), the same argument as in the proof of Proposition \ref 
{prop:multiplebranch} shows that $R$ is smooth at the two points $x_1,x_2$ over $y$. 
The lines in the  tangent cone to $B$ at $y$ consist of the images, via the projection
from $p$, of the two tangent planes $T_{S,x_i}$ to $S$ at the branch corresponding to 
$x_i$, $i=1,2$.  We claim that these two planes are distinct, then also their projections
from $p$ are distinct, thus $y$ is a node for $B$. 

To prove this, consider the closure
$W$ in $\Sigma\times \Sigma$
of the pairs $(x_1,x_2)$ of distinct, smooth points, such that  $T_{S,x_1}= T_{S,x_2}$. 
One has $\dim(W)\leq 1$. In fact  the dual $\Sigma^ *$ of $\Sigma$ is a surface because
$\Sigma$ is not a developable surface (see \cite {GH}), and the points in $W$ correspond to singular points of $\Sigma^ *$. Then, by the genericity of $p$, there is no pair of points
 $(x_1,x_2)$ in $W$ such that $p, x_1,x_2$ are collinear, which proves our claim. 
 
 In case (b), we will show, with a direct computation, that $R$ is smooth at $x:=x_1$,
 it is tangent there to $\ell$, and the image of the tangent plane $T_{S,x}$ via the projection
 from $p$ has intersection multiplicity 3 with $B$ at $y$. This will prove that $y$ is a cusp.
 
Choosing affine coordinates, we may assume that $p$ is the point at infinity of the $z$--axis,
that $\nu(x)$ is the origin, that the tangent plane to $\Sigma$ at $z$ is the plane $y=0$,
that the asymptotic tangent lines to $\Sigma$ at the origin are the $z$ and the $y$--axes;
note that these asymptotic lines are distinct by Lemma \ref {lem:one}. 

In this coordinate system $\Sigma$ has equation of the form

\begin{equation}\label{eq:effe}
f_0(x,z)+y+\sum_{i=2}^ dy^ if_i(x,z)=0
\end{equation}
where $f_i(x,z)$, $i \geq 2$, is a polynomial of degree at most $d-i$ and 

\begin{equation}\label {eq:ici}
f_0(x,z)=axz+a_1x^ 2+a_2x^ 2z+a_3xz^ 2+a_4z^ 3+o(3)\end{equation}
with $a\cdot  a_4\neq 0$. 

In this setting, the ideal of $R$  around the origin is generated by the first
member of  \eqref {eq:effe}
and by its derivative with respect to $z$, i.e.

\begin{equation}\label{eq:effez}
{\frac {\partial f_0} {\partial z}} (x,z)+\sum_{i=2}^ dy^ i  \frac  {\partial f_i}  {\partial z}  (x,z).
\end{equation}
To prove that $R$ is smooth at the origin, one has to prove that the plane $y=0$ is not
tangent at the origin  to the surface defined as the zero locus of \eqref {eq:effez}, i.e.
that the curve with equation

$${\frac {\partial f_0} {\partial z}} (x,z)=0$$
is not singular at the origin. This is immediate, since the curve in question
has equation

$$ax+a_2x^ 2+2a_3xz+3a_4z^ 2+o(2)=0$$
with $a\neq 0$. Note that we can write a local analytic equation of
this curve as

\begin{equation}\label {eq:cic}
x=-\frac {3a_4}a z^ 2+o(2).\end{equation}

Next, let us compute the intersection multiplicity of
the tangent plane to $\Sigma$ at the origin, with $R$. This amounts to compute 
the intersection multiplicity at the origin of the two curves

$$f_0(x,z)=0, \quad {\frac {\partial f_0} {\partial z}}  (x,z)=0.$$
This is the order at 0 of the power series obtained by substituting \eqref {eq:cic} into
\eqref  {eq:ici}, which is clearly 3.  This implies that the line $z=0$ has multiplicity of
intersection 3 with $B$ at the origin, proving that $B$ has a cusp there. 

Finally we have to prove the assertion about the irreduciblity of $R$. 
Since, as we saw,  $R$ is smooth, it suffices to prove that it is connected.
Let $H$ be a curve
in $\mathcal L$.  By the Riemann-Hurwitz theorem we have 
$R\sim K_X +3 H$, where $K_X$, as usual, denotes a canonical divisor of $X$ and $\sim$ denotes linear equivalence. 
By adjunction theory (cf.\ e.g.
\cite{Ionescu} and \S\ 7 in \cite{CR}) one has that $K_X + 2H$ is nef and
since $H$ is ample, then $R$ is ample, hence it is connected, finishing the proof.
\end{proof}

\section{Focal loci}\label{S:fuochi}

In this section, we briefly recall some basic definitions and results concerning the so called 
{\em focal loci} of families of projective varieties. These will be  essential in the next section. 
We follow  \cite{CC} and \cite{CS}, inspired in turn by  \cite{Seg}.

\subsection {The focal machinery}\label {SS:focal}
Let $Y$ be a smooth, irreducible, projective variety and let 
\begin{equation}\label{eq:focald}
\begin{array}{ccl}
\X
 & \subset & D \times Y \\
\downarrow & & \\
D & & 
\end{array}
\end{equation}be a flat family of closed subschemes of $Y$ parametrized by the base scheme $D$, which we assume to be 
integral. 
Denote by 
$$q_1 : \X
 \to D \;\;{\rm and}\;\; q_2: D \times Y \to Y$$the natural projections. Set $f:= q_2|_{\X
}$. For every point $z\in D$ we denote by $X_z$ the fibre of $q_1$ over $z$.

For any scheme $Z$, set 
$$\T_Z := {\mathcal Hom}(\Omega^1_{Z}, \Oc_Z)$$ and let 
$$\N := \N_{\X
/B\times Y}\;\; {\rm and} \;\; 
\T(q_2) : = {\mathcal Hom}(\Omega^1_{B \times Y /Y}, \Oc_{B \times Y}).$$One has the following 
commutative diagram of sheaves on $\X
$
\begin{equation}\label{eq:focaldiag}
\begin{array}{ccccccc}
  &  & & \T(q_2)|_{\X
} & \stackrel{\chi}{\to} & \N & \\
 &  & & \downarrow & & || & \\
0 \to & \T_{\X
} & \to & \T_{D \times Y}|_{\X
} & \to & \N & \to 0 \\
 &  \downarrow^{df} & & \downarrow & &  & \\
  & q_2^*(\T_Y|_{\X
}) & =  & q_2^*(\T_Y|_{\X
}) & & & 
\end{array}
\end{equation}

\noindent
called the {\em focal diagram} of the family \eqref{eq:focald} (cf. \cite[Diagram (3)]{CS}).

The map $\chi$ is defined by the commutative diagram 
and is called the {\em global characteristic map} of the family \eqref{eq:focald}. From the focal diagram \eqref{eq:focaldiag}, one sees that 
${\rm ker} (\chi) = {\rm ker} (df)$.
We denote by $\Ff$ this sheaf.
We will mainly consider the case in which $f: \X
\to Y$ is dominant and generically finite, so that
$\Ff$ is a torsion sheaf which we call the \emph{focal sheaf} of the family \eqref {eq:focald}.
Its support
$\Ff(\X
)$ is called the \emph{focal scheme} of the family, 
and 
$\dim (\Ff(\X
))<\dim (\X
)$. 
If $z\in D$ is a point, we denote by $\Ff(X_z)$ the intersection of the focal scheme with 
$X_z$.

 From the focal diagram \eqref{eq:focaldiag}, one can think of 
$\Ff(\X
)$ as the set of ramification points of the map $f$. We denote by $L(\X)$ its image
via $f$, i.e. the set of branch points of  $f$. 

\subsection {Filling families of linear spaces}\label {SS:filling}
The situation to have in mind for our applications is the following: $Y = \Pp^r$ and $\X
$ is a family of $k$--dimensional linear subspaces of $\Pp^r$ such that 
$\dim(D) = r-k$ and $f:\X
\to \Pp^ r$ is dominant. This is called a \emph{filling family} 
of linear subspaces of $\Pp^ r$.
For example,  $D$ could be a $(r-k)$--dimensional subscheme of the Grassmannian 
$\mathbb{G}(k,r)$, with the filling property. 

Since we will work at the general point of $D$, and since $D$ is integral, we may and will assume that $D$ is smooth.  

\begin{proposition}\label{prop:tuttofuoco} If  \eqref {eq:focald} is a filling family of 
 $k$--dimensional linear subspaces of $\Pp^r$ and if $z$ is a general point of $D$, then
 $\Ff(X_z)$ is a 
hypersurface of degree $r-k$ in $X_z\simeq \Pp^ k$. 
\end{proposition}

\begin{proof} If one restricts the global characteristic map $\chi$ to the fibre $X_z$ 
this reduces to 
\begin{equation}\label{eq:caratt}
 \Oc_{X
_z}^{\oplus(r-k)} \cong T_{D,z} \otimes \Oc_{X
_z}  \stackrel{\chi_z}{\longrightarrow} \N_{X
_z/\Pp^r} 
\cong \Oc_{X
_z}(1)^{\oplus(r-k)}.
\end{equation}In particular, for any $z \in D$, the map $\chi_z$ can be viewed 
as a square matrix $A_z$ of size $r-k$, with linear entries. Thus, $\Ff(X
_z)$
is defined by the equation $\det(A_z) = 0$ and the assertion follows. Note that $\det(A_z)$
cannot be identically zero, because $\dim(\Ff(\X))<\dim(\X)$, and therefore also
$\dim (\Ff(X_z))<\dim (X_z)=k$. 
\end{proof}

\subsection {Filling families of lines in $\Pp^ 3$}\label {SS:lines}
More specifically, we will consider filling families of lines in $\Pp^ 3$. In this case 
on the general line of the family there are two foci, which can either
 be distinct or only one with multiplicity 2. It is useful to recall how the equation of 
 foci can be computed on the general line  $\ell$ of the family $\X$. 
 
Since the problem is local, 
 we may assume the family to be parametrized by 
 a bidisc $D$. 
 More precisely, if $z=(u,v)$ is a point of $D$, 
 we may assume that the line $\ell_z$ is described as the intersection of the
  two planes with equation
 \begin{equation}\label {eq:retta}
 a(z)\times x=b(z)\times x=0,\end{equation}
 with 
 $$a(z)=(a_0(z),a_1(z),a_2(z),a_3(z)), \quad b(z)=(b_0(z),b_1(z),b_2(z),b_3(z)), \quad x=(x_0,x_1,x_2,x_3).$$ We may 
write $a, b, a_i, b_i$ rather that $a(z), b(z), a_i(z), b_i(z)$. We will denote with lower case indices the derivatives with
respect to the variables $u,v$, i.e. $a_u=\frac {\partial a}{\partial u}$,   $a_{i,u}=\frac {\partial a_i}{\partial u}$. etc.

In this setting, the characteristic map can be described by looking at \eqref {eq:caratt}. One sees that the equation of the focal locus on $\ell_z$ is
 
 \begin{equation}\label{eq:focaleq2}
 \det \left(\begin{matrix} a_u\times x&  a_v\times x\cr
 b_u\times x&  b_v\times x\cr
 \end{matrix} \right)=0
\end{equation}
modulo \eqref {eq:retta}.

\section{Filling families of tangent lines to a surface}\label{S:fillingfam}

 In this section, using focal techniques, we prove a proposition 
 which is an essential tool for the proof of Theorem \ref{thm:main}.
 This proposition was certainly known to
the  classical algebraic and projective--differential geometers.
 Since however  we do not have
a suitable reference for it, we give here its complete proof. 

\begin{proposition} \label {prop:foctang} Let $\X$ be a filling  family of lines in $\Pp^ 3$.
Assume that its general member $\ell$ is tangent to a non--developable surface $\Sigma$ at a general point $p$ of it.
Then:
\begin{itemize}
\item [(a)] $p$ is a focus on $\ell$;
\item  [(b)] the contact order of $\ell$ with $\Sigma$ at $p$ is at most $2$;
\item [(c)] if the contact order of $\ell$ with $\Sigma$ at $p$ is $2$, then $p$ is
a focus with multiplicity two on $\ell$.
\end{itemize}
\end{proposition}

\begin{proof} The question being local, we may assume that $\X$ is given around $\ell$ as follows. Let $\Sigma$ be locally parametrized around $p$ as $p=p(u,v)$, with $z=(u,v)\in D$, where $D$ is a bidisc.
Then $\ell_z$ is defined by the equations \eqref{eq:retta}, where we may assume that the plane $a\times x=0$ is tangent to $\Sigma$ at $p$, i.e. one has
\begin{equation}\label {eq:ooo}
a\times p=a\times p_u=a\times p_v=0.
\end{equation}
By differentiating the first relation in \eqref{eq:ooo} and taking into account the other two, we find
\begin{equation}\label {eq:rel1}
a_u\times p=a_v\times p=0.
\end{equation}
Taking into account equation \eqref {eq:focaleq2}, (a) immediately follows.

Before proceeding, note that the dual variety $\Sigma^ *$ of $\Sigma$ is a surface, since we are assuming that $\Sigma$ is not developable (see \cite {GH}). This implies that $a, a_u$, and
 $a_v$ are linearly independent.
 
Assume now $\ell$ is an asymptotic tangent line to $\Sigma$ at $p$. We may suppose that $\ell=\ell_z$
has the tangent direction of the vector $p_u$ at $p$. This translates into the relation
\begin{equation}\label{eq:asympto2}
a\times p_{uu}=0
\end{equation}
and $\ell_z$ is parametrically described by 
\begin{equation}\label{eq:asympto}
x=x_0p+x_1p_u.\end{equation}
By differentiating the second equation in \eqref  {eq:ooo} and
taking into account \eqref {eq:asympto2}, we find
\begin{equation}\label {eq:aaa}
a_u\times p_u=0.\end{equation}
Hence the plane $a_u\times x=0$ is distinct from
the tangent plane $a\times x=0$ and contains the line $\ell$. So
$\ell$ is defined by the equations
$$a\times x=a_u\times x=0$$
and the equation \eqref {eq:focaleq2} of the focal locus
becomes
\begin{equation}\label{eq:focaleq2bis}
 \det \left(\begin{matrix} a_u\times x&  a_{v}\times x\cr
 a_{uu}\times x&  a_{uv}\times x\cr
 \end{matrix} \right)=0
\end{equation}

By differentiating the first equation in \eqref  {eq:rel1} and
considering \eqref {eq:aaa}, we get 
\begin{equation}\label {eq:ccc}
a_{uu}\times p=0.\end{equation}
Substituting \eqref {eq:asympto} into \eqref {eq:focaleq2bis}, and taking into account 
\eqref  {eq:rel1}, \eqref  {eq:aaa} and \eqref {eq:ccc},
we find the equation
\begin{equation}\label{eq:focaleq3}
x_1^ 2(a_v\times  p_u)(a_{uu}\times p_u)=0
\end{equation}
which, by the filling property, is not identically zero (see Proposition \ref {prop:tuttofuoco}).

Note that $a_v\times  p_u\neq 0$, because, as we saw, $a_v$ is linearly independent 
from $a$ and $a_u$. In addition we have $a_{uu}\times p_u\neq 0$.
On the other hand, by differentiating \eqref {eq:asympto2}  and \eqref  {eq:aaa} we see that
$$a_{uu}\times p_u=a\times p_{uuu}.$$
Thus one has $a\times p_{uuu}\neq 0$ which proves (b). Summing up, the equation \eqref {eq:focaleq3}
becomes $x_1^ 2=0$, proving (c).
\end{proof}

\begin {remark}\label {rem:coont} {\rm As a consequence of  Proposition \ref {prop:foctang}, 
there is no non--developable surface $\Sigma$ in $\Pp^ 3$ having a 2--dimensional family $\X$ of  non--simple asymptotic tangent lines. The same holds if $\Sigma$ is developable but not a plane. We do not dwell on this now.

Conversely, if $\X$ is a filling family of lines in $\Pp^ 3$, which does not have 
\emph{fundamental points}, i.e. points $p\in \Pp^ 3$ contained in infinitely many lines of $\X$, then 
the focal locus of  $\X$ in $\Pp^ 3$ is a, may be reducible, surface $\Sigma$ and $\X$ is formed by lines which are either bitangents or asymptotic tangents to $\Sigma$. We do  not dwell on this as well.
}\end{remark}

\section{The proofs of the main  theorems}\label{S:Main}

We are now in position to give the:

\begin {proof} [Proof of Theorem \ref {thm:main2}] By Proposition \ref {prop:doublebranch}, it suffices
to prove that $B$ has only double points. In view of Proposition \ref  {prop:multiplebranch}, one has
to show that there is no filling family of lines whose general member $\ell$ is such that 
$\ell_X=n_1x_1+\ldots+n_hx_h$ with $\sum_{i=1}^ h(n_i-1)\geq 3$. This is ensured by \S\, \ref{SS:filling} and by  Proposition \ref  {prop:foctang}. \end{proof} 

Finally, we have the:

\begin{proof}  [Proof of Theorem \ref {thm:main}] It follows from  Theorem \ref {thm:main2} and from the General Projection  Theorem \ref {thm:genpro}.

\end{proof}

\begin{remark}\label{rem:kulkul} {\normalfont Suppose that $\Sigma$ has, off the ordinary singularity locus, 
a finite number of further double points where the germ of $\Sigma$ is analytically equivalent to the one of  
an affine surface in $\CC^3$ with equation $z^2 = h(x,y)$, at the origin, where $h(x,y) = 0$ is a curve sigular 
at the origin. As in \cite{kulkul}, in particular see Lemma 1.4, one proves that the branch curve $B$ 
of a general projection of $\Sigma$ to a plane has again only nodes and cusps besides the singularities 
arising from the projections of the aforementioned double points, where the singularity of $B$ 
is locally analytically  equivalent to the one of the curve $h(x,y) = 0$. This extension of Theorem 
\ref{thm:main2} implies an obvious analogous extension of Theorem \ref{thm:main}. Note however 
that the irreducibility 
statement abount branch and ramification curves may file in this situation. 
}
\end{remark}

\section{Applications to Hilbert scheme dimension computations}\label{S:considerations}

As mentioned in the introduction, Theorem \ref   {thm:main} has  important applications
in the theory of surfaces. 
In this section we recall one, namely Enrique's approach to the computation
of the dimension of Hilbert schemes and moduli spaces of surfaces (see \cite {enr}).

Let $S\subset \Pp^ r$ and 
$\varphi:= \varphi_2 : S \to \Pp^2$ be as usual.
One has the exact sequence
\begin{equation}\label{eq:tgS4}
0 \to T_S \stackrel{d\varphi}{\longrightarrow} \varphi^*(T_{\Pp^2}) \to \N_{\varphi} \to 0,
\end{equation}
defining  $ \N_{\varphi}$, which is called 
the {\em normal sheaf} to the map $\varphi$, fitting also in the 
so called {\em Rohn exact sequence}
\begin{equation}\label{eq:tgS5}
0 \to \oplus_{i=1}^{r-2} \Oc_S(H) \to  \N_{S/\Pp^{r}} \to \N_{\varphi} \to 0
\end{equation}(see, e.g. \cite[p. 358, formula (2.2)]{Cil}). 

The sheaf $ \N_{\varphi}$ has torsion, being supported on $R \subset S$, the ramification
locus of $\varphi$. The morphism $\phi := \varphi|_{R} : R \to B$ is birational onto the branch curve, 
which has only nodes and cusps as singularities by Theorem \ref{thm:main}, and
$$\deg(R)=\deg(B)=2(d+g-1).$$
Let $G \subset R$ be the divisor formed by all points $p$ such that $\phi(p)$ is
a cusp of $B$, each counted
with multiplicity one, and let $i : R \hookrightarrow S$ be the inclusion of $R$ in $S$.
By the exact sequence \eqref{eq:tgS4} and the analogous one for the map $\phi$, we get the commutative diagram

\begin{equation}\label {eq:diag1}
\begin{array}{rcccccl}
&  &  &  &  & 0 & \\
&  &  &  &  & \downarrow & \\
 & 0 &  &  &  & \mathcal G \cong \Oc_G & \\
&  \downarrow  &  &  &  & \downarrow & \\
0 \to & T_{R} & \stackrel{d \phi}{\longrightarrow} & \phi^*(T_{\Pp^2}) & \to & \N_{\phi} & \to 0 \\
 & \downarrow & & ||&  & \downarrow &  \\
  & i^*(T_S) & \stackrel{d \varphi}{\longrightarrow} & \phi^*(T_{\Pp^2}) & \to & i^*(\N_{\varphi}) & \to 0 \\
  &  & & &  & \downarrow &  \\
 &  & & &  & 0 &
\end{array}
\end{equation}
Consider now the following 
diagram (see \cite[p. 24]{AC}):

\begin{equation}\label {eq:diag2}
\begin{array}{rcccccl}
&  &  &  &  & 0 & \\
&  &  &  &  & \downarrow & \\
 & 0 &  &  &  & \mathcal G \cong \Oc_G & \\
&  \downarrow  &  &  &  & \downarrow & \\
0 \to & T_{R} & \stackrel{d \phi}{\longrightarrow} & \phi^*(T_{\Pp^2}) & \to & \N_{\phi} & \to 0 \\
 & \downarrow & & ||&  & \downarrow &  \\
  & T_{R} \otimes \Oc_{R} (G) & \stackrel{d \phi}{\longrightarrow} & \phi^*(T_{\Pp^2})
& \to & \N_{\phi}'  & \to 0 \\
  &  & & &  & \downarrow &  \\
 &  & & &  & 0 &
\end{array}
\end{equation}
where $\N_{\phi}' $ is a line bundle on $R$. 
Note that these two diagrams imply 
\begin{equation}\label {eq:equal}
h^ j(S,\N_\varphi)=h^ j(R,i^ *(\N_\varphi))=h^j(R,\N'_\phi), \quad j=0,1,2.
\end{equation}

One has
$${\rm deg}(\N_{\phi}) = \deg (\phi^*(T_{\Pp^2})) - {\rm deg}(T_{R})= 3 {\rm deg}(R) + \deg(K_R)=
6(d+g-1)+ \deg(K_R).$$
 The classical formula for the number of cusps of the branch curve, in case this has only nodes and cusps, gives
 $$
{\rm deg}(Z) = 3(d+K_S^2 - 4 \chi(\Oc_S) +6(g-1))$$
(cf. formulas (a), (b) and (d) 
in \cite[Prop. 2.6]{CMT}; cf. standard references as \cite[p. 182]{enr} and \cite{I}).

Thus
$${\rm deg}(\N_{\phi}') = \deg (\N_{\phi}) - {\rm deg}(Z)=
3\left(d-K_S^ 2-4(g-1)+4\chi(\Oc_S) \right)+\deg(K_R).$$Since 
\begin{equation}\label{eq:paR}
p_a(R)=9(g-1)+K_S^ 2+1
\end{equation}one has  
\begin{equation}\label {eq:equal2}
h^ 0(S,\N_\varphi)=h^ 0(R, \N'_\phi)=3d-3(g-1)-2K_S^ 2+12\chi(\Oc_S)+h
\end{equation}
where $h=h^ 1(R,\N'_\phi)=h^ 1(S,\N_\varphi)$. Note that
$\N_{\phi}'$ is  
non-special, hence $h=0$, if
\begin{equation}\label {eq:nonspec}
d-K_S^ 2-4(g-1)+4\chi(\Oc_S)>0.
\end{equation}Otherwise, in case $\N_{\phi}'$ is special, note that $R$ is not hyperelliptic, 
since $\Oc_R(H)$ is clearly special. Then, taking into account \eqref{eq:paR}, one has 
\begin{equation}\label {eq:bound}
h\leq \frac{3}{2}(K_S^2 - d) + 6(g-1 - \chi(\Oc_S)) + \frac{1}{2},
\end{equation}
by Clifford's theorem. More precisely, 

Let now 
$$
h^1(S, \Oc_S(H))  = h(S)
$$
be the {\em speciality} of $S$. 

From  \eqref{eq:tgS5}, \eqref{eq:diag1}, \eqref{eq:diag2}, \eqref {eq:equal}, \eqref {eq:equal2} and \eqref{eq:bound} 
we conclude that:

\begin{theorem} In the above setting, one has
$$
h^0(S, \N_{S/\Pp^{r}}) \leq (r-2)(r+1) + h^0(R,\N'_{\phi}) =$$
$$ (r-2)(r+1) + 3d-3(g-1)-2K_S^ 2+12\chi(\Oc_S)+h.$$
If  $\N_{\phi}'$ is non-special (in particular if \eqref {eq:nonspec}
holds) one has:
\begin{itemize}
\item [(i)]  $h^1(S, \N_{S/\Pp^{r}})  \leq (r-2) h(S),$
\item [(ii)] $h^2(S,\N_{S/\Pp^{r}}) = (r-2) \, h^2(\Oc_S(H)).$
\end{itemize}
The above inequalities become equalities if and only if the map $H^0(\N_{S/\Pp^{r}}) \to H^0(\N_{\varphi})$ 
arising from \eqref{eq:tgS5} is surjective. This is the case if $h(S)=0$.

\noindent
If  $\N_{\phi}'$ is special, one has 
$$h^0(S, \N_{S/\Pp^{r}})\leq (r^ 2-r-1) + \frac 12(3d - K_S^ 2-1) + 3 (g-1)+ 6 \chi(\Oc_S).$$
\end{theorem}

\begin{remark}\label {rem:canonicalbis} {\rm Suppose that $H\sim K_S$, i.e. that the canonical system is very ample. 
Let $q(S) = h^1(S, \Oc_S)$ be, as usual, the {\em irregularity} of $S$. 
The above theorem implies that  the number $M(S)$ of moduli of $S$ is bounded above by 
$4K_S^ 2 + 3 \chi(\Oc_S) - 3 q(S) + 4$, if $h>0$, and by $12\chi(\Oc_S) - 3 p_g- 1 - 2 K_S^2$. 
In the latter case, since $K_S^2 \geq 3p_g - 7$ by Castelnuovo's inequality (see e.g. \cite{bpv}), 
one has $M(S) \leq 3p_g - 12 q + 25$.    
For the problem of finding good upper bound for the number of moduli of a surface, see \cite {cat1}, \cite {cat2}.
}\end{remark}




\begin{thebibliography}{ADSE}

\bibitem{AC} E.~Arbarello, M.~Cornalba, Su una congettura di Petri, {\em Comment. Math. Helvetici},
{\bf 56}(1981), 1-38.

\bibitem {bpv} W. ~Barth,  K. ~Hulek,  C.~Peters,  A.~Van de Ven, {\em Compact Complex Surfaces}, 
Ergebnisse der Mathematik und ihrer Grenzgebiete. {\bf 3}. Folge. A Series of Modern Surveys in Mathematics, 4. Springer-Verlag, Berlin, 2004. 


\bibitem {cat1} F.~ Catanese, {\em Moduli of surfaces of general type}, in "Algebraic Goemtry -- Open Problems", Proceedings of the Ravello Conference, 1982, Springer Lecture Notes in Math., 997 (1983), 90--112. 

\bibitem {cat2} F.~ Catanese, {\em On the moduli spaces of surfaces of general type}, Jour. of Diff. Geom., 
{\bf 19} (1984), no. 2, 483--515. 



\bibitem{CC} L.~Chiantini, C.~Ciliberto, A few remarks on the lifting problems, Journ\'ees de g\'eom\'etrie alg\'ebrique d'Orsay, {\em Ast\'erisque}, {\bf 218} (1993), 95-109. 

\bibitem{CC2} L.~Chiantini, C.~Ciliberto, Weakly defective varieties, {\em Transactions of AMS}, {\bf 354} (2001), 151--178.  



\bibitem{Cil} C.~Ciliberto, On the Hilbert scheme of curves of maximal genus in a projective space,
{\em Math. Z.}, {\bf 194} (1987), 351--363.


\bibitem{CMT} C.~Ciliberto,  R.~Miranda, M.~Teicher,
Pillow degenerations of $K3$ surfaces, {\em Applications of algebraic geometry to coding theory, phisics and computations}, (Eilat 2001), 53--63, NATO Sci. Ser. II Math. Phys. Chem., {\bf 36}, Kluwer Acad. Publ., Dordrecht, 2001. 

\bibitem{CR} C.~Ciliberto, F.~Russo, Varieties with minimal secant degree and linear systems
of maximal dimension on surfaces, to appear in
{\em Adv. Math.}.


\bibitem{CS} C.~Ciliberto, E.~Sernesi, Singularities of the theta divisor and congruences of planes, 
{\em J. Algebraic Geom.}, {\bf 1} (1992), 231-250. 

\bibitem{enr} F.~Enriques, {\em Le superficie algebriche}, Zanichelli, Bologna, 1949.


\bibitem{Fal} G.~Faltings, A new application of Diophantine approximations. {\em 
A panorama of number theory or the view from Baker's garden} (Z\"urich, 1999), 231--246, Cambridge Univ. Press, Cambridge, 2002. 


\bibitem{Fra} A.~Franchetta, Sulla curva doppia della proiezione della superficie generale dell'$S_4$, da un punto generico su un $S_3$, {\em Rend. Acc. d'Italia} Ser. VII {\bf 2} (1940), 282--288 and {\em Lincei - Rend. Sc. fis. mat. nat.} {\bf 2} (1947), 276--279.  

\bibitem {Ful} W. ~Fulton, \emph{Intersection theory}, Springer Verlag, 1998.


\bibitem {GH2} P. ~Griffiths, J. ~Harris, {\em Principles of algebraic geometry}, 
Pure and Applied Mathematics. Wiley-Interscience, John Wiley and Sons, New York, 1978. 


\bibitem {GH} P. ~Griffiths, J. ~Harris, {\it Algebraic geometry and
local differential geometry}, Ann. Sci. Ecole Norm. Sup., {\bf 12}
(1979), 355-432.

\bibitem{Ionescu} P.~Ionescu, Generalized adjunction and applications, {\em Math. Proc. Camb. Phil. Soc,}
{\bf 99} (1986), 467--472.

\bibitem{I} B.~Iversen, Critical points of an algebraic function,
{\em Invent. Math.}, {\bf 12} (1971), 210--224.



\bibitem {kul1} V.~Kulikov, {\em Generic coverings of the plane and braid monodromy invariants}, 
The Fano Conference, 533--558, Univ. Torino, Turin, 2004.

\bibitem {kul2} V.~Kulikov, {\em Hurwitz curves}, 
Russian Math. Surveys, {\bf 62} (6), 1043--1119.

\bibitem {kulkul} V.~Kulikov, V. S.~Kulikov, {\em Generic coverings of the plane with A-D-E-singularities}, Izvestiya: Mathematics, {\bf 64} (6), 1153--1195.


\bibitem{MP} E.~Mezzetti, D.~Portelli, 
A tour through some classical theorems on algebraic surfaces. {\em An. Stiint. Univ. Ovidius Constanta Ser. Mat.},  
{\bf 5} (1997), no. 2, 51--78. (cf. also http://it.wikipedia.org/wiki/Teorema$\!_{-}$di$\!_{-}$Kronecker-Castelnuovo)


\bibitem{Mum} B.~Moishezon,  {\em Complex surfaces and connected sums of complex projective planes}.
Lecture Notes in Mathematics, {\bf 603}. Springer-Verlag, New York, 1977.





\bibitem{Seg} C.~Segre, Sui fuochi di $2^o$ ordine dei sistemi infiniti di piani e sulle curve iperspaziali con una doppia infinit\`a di piani plurisecanti, {\em Atti R. Accad. Lincei}, (5) {\bf 30} (1921), 67--71.



\bibitem{Ser} E.~Sernesi, {\em Deformations of Algebraic Schemes}, 
Grundlehren der Mathematischen Wissenschaften {\bf 334}. Springer-Verlag, Berlin, 2006. 

\bibitem{Sev} F.~Severi, Intorno ai punti doppi impropri di una superficie generale dello spazio a quattro dimensioni, e ai suoi punti tripli apparenti, {\em Rend. Circolo Matematico di Palermo}, {\bf 15} (1901), 33--51; also in 
{\em Opere Matematiche}, {\bf I}, VI, 14-30.  

\bibitem {teic} M. ~ Teicher, {\em 
Braid Group, Algebraic Surfaces and Fundamental Groups of complements of Branch Curves}, Algebraic Geometry - Santa Cruz 1995, 127--150, {\em Proc. Sympos. Pure Math}, {\bf 62}, part 1, Amer. Math. Soc., Providence, RI, 1997. 


\end{thebibliography}
\end{document}